\documentclass{patmorin}
\usepackage{pat}
\usepackage{paralist}
\usepackage[OT1]{fontenc}
\usepackage[utf8x]{inputenc}
\usepackage{paralist}
\usepackage{bbm}  %

\usepackage{todonotes}

\usepackage{etoolbox}

\usepackage[longnamesfirst,numbers,sort&compress]{natbib}

\usepackage[mathlines]{lineno}
\newcommand*\patchAmsMathEnvironmentForLineno[1]{%
 \expandafter\let\csname old#1\expandafter\endcsname\csname #1\endcsname
 \expandafter\let\csname oldend#1\expandafter\endcsname\csname end#1\endcsname
 \renewenvironment{#1}%
    {\linenomath\csname old#1\endcsname}%
    {\csname oldend#1\endcsname\endlinenomath}}%
\newcommand*\patchBothAmsMathEnvironmentsForLineno[1]{%
 \patchAmsMathEnvironmentForLineno{#1}%
 \patchAmsMathEnvironmentForLineno{#1*}}%
\AtBeginDocument{%
\patchBothAmsMathEnvironmentsForLineno{equation}%
\patchBothAmsMathEnvironmentsForLineno{align}%
\patchBothAmsMathEnvironmentsForLineno{flalign}%
\patchBothAmsMathEnvironmentsForLineno{alignat}%
\patchBothAmsMathEnvironmentsForLineno{gather}%
\patchBothAmsMathEnvironmentsForLineno{multline}%
}

\newcommand{\vol}[1]{\boxplus_{#1}}

\DeclareMathOperator{\interior}{Int}

\definecolor{brightmaroon}{rgb}{0.76, 0.13, 0.28}
\definecolor{linkblue}{rgb}{0, 0.337, 0.227}
\newcommand{\defin}[1]{\emph{\color{brightmaroon}#1}}
\DeclareMathOperator{\tw}{tw}

\DeclareMathOperator{\chicen}{\chi_{\mathrm{cen}}}
\DeclareMathOperator{\chilin}{\chi_{\mathrm{lin}}}

\DeclareMathOperator{\binomial}{binomial}

\newrobustcmd{\onesub}{\mathord{\includegraphics{figs/one-sub}}}
\newrobustcmd{\leftup}{\mathord{\includegraphics{figs/left-up}}}

\title{\MakeUppercase{Linear versus centred chromatic numbers}\thanks{This research was partly funded by NSERC.}}
\author{Prosenjit~Bose\thanks{School of Computer Science, Carleton University},\quad Vida~Dujmović\thanks{Department of Computer Science and EE, University of Ottawa},\quad Hussein~Houdrouge\footnotemark[2],\quad Mehrnoosh~Javarsineh\footnotemark[2],\quad and Pat~Morin\thanks{School of Computer Science, Carleton University, \email{morin@scs.carleton.ca}}}

\DeclareMathOperator{\VE}{\mathit{VE}}

\date{}

\begin{document}

\maketitle
\renewcommand{\E}{\mathbb{E}}
\renewcommand{\Pr}{\mathbb{P}}

\begin{abstract}
  A \emph{centred colouring} of a graph is a vertex colouring in which every connected subgraph contains a vertex whose colour is unique and a \emph{linear colouring} is a vertex colouring in which every (not-necessarily induced) path contains a vertex whose colour is unique.  For a graph $G$, the \emph{centred chromatic number} $\chicen(G)$ and the \emph{linear chromatic number} $\chilin(G)$ denote the minimum number of distinct colours required for a centred, respectively, linear colouring of $G$. From these definitions, it follows immediately that $\chilin(G)\le \chicen(G)$ for every graph $G$. The centred chromatic number is equivalent to treedepth and has been studied extensively. Much less is known about linear colouring.  Kun \etal\ [\textit{Algorithmica} \textbf{83}(1)] prove that $\chicen(G) \le \tilde{O}(\chilin(G)^{190})$ for any graph $G$ and conjecture that $\chicen(G)\le 2\chilin(G)$.  Their upper bound was subsequently improved by Czerwinski \etal\ [\textit{SIDMA} \textbf{35}(2)] to $\chicen(G)\le\tilde{O}(\chilin(G)^{19})$. The proof of both upper bounds relies on establishing a lower bound on the linear chromatic number of pseudogrids, which appear in the proof due to their critical relationship to treewidth.  Specifically, Kun \etal\ prove that $k\times k$ pseudogrids have linear chromatic number $\Omega(\sqrt{k})$. Our main contribution is establishing a tight bound on the linear chromatic number of pseudogrids, specifically $\chilin(G)\ge \Omega(k)$ for every $k\times k$ pseudogrid $G$. As a consequence we improve the general bound for all graphs to $\chicen(G)\le \tilde{O}(\chilin(G)^{10})$. In addition, this tight bound gives further evidence in support of Kun \etal's conjecture (above) that the centred chromatic number (i.e., the treedepth) of any graph is upper bounded by a linear function of its linear chromatic number.
\end{abstract}

\section{Introduction}

Let $G$ be a simple undirected graph.  A \defin{$k$-colouring} of $G$ is any function $\varphi:V(G)\to C$ where $C$ is a set of size $k$.  A vertex $v$ of $G$ is a \defin{centre} of $G$ with respect to $\varphi$ if $\varphi(v)\not\in\{\varphi(w):w\in V(G)\setminus\{v\}\}$, i.e., $v$ is the unique vertex of $G$ having colour $\varphi(v)$.  A colouring $\varphi$ of $G$ is \defin{centred} if every connected subgraph of $G$ has a centre with respect to $\varphi$. A colouring $\varphi$ of $G$ is \defin{linear} if every path\footnote{A \defin{path} in a graph $G$ is a sequence of distinct vertices $v_0,\ldots,v_r$ such that $v_{i-1}v_i$ is an edge of $G$ for each $i\in\{1,\ldots,r\}$.} in $G$ has a centre with respect to $\varphi$.  The \defin{centred chromatic number} $\chicen(G)$ of $G$ is the minimum integer $c$ such that $G$ has a centred $c$-colouring and the \defin{linear chromatic number} $\chilin(G)$ is the minimum integer $c$ such $G$ has a linear $c$-colouring.  Since each edge of $G$ is both a connected subgraph and a path in $G$, any centred colouring and any linear colouring is a proper colouring of $G$.

The centered chromatic number of any graph $G$ is equal to the treedepth of $G$ and has been studied extensively \cite{nesetril.ossona:sparsity,nesetril.ossona:on,nesetril.ossona:grad,nesetril.ossona:tree-depth,pilipczuk.siebertz:polynomial,dereniowski.nadolski:vertex,dereniowski.kubale:cholesky,deogun.kloks.ea:on,bodlaender.gilbert.ea:approximating,bodlaender.deogun.ea:rankings}. Much less is known about the linear chromatic number.  Linear chromatic number was introduced by \citet{kun.obrien.ea:polynomial} who were motivated by finding efficiently-computable approximations of treedepth in bounded expansion classes.  Since every path in $G$ is a connected subgraph of $G$, every centred colouring of $G$ is also a linear colouring of $G$, so $\chilin(G)\le\chicen(G)$. In the other direction, \citet{kun.obrien.ea:polynomial} were able to establish that $\chicen(G)\le \chilin(G)^{190}\cdot(\log(\chilin(G)))^{O(1)}$.  This upper bound was subsequently improved by \citet{czerwinski.nadara.ea:improved}, who reduced the exponent to $19$.

\begin{thm}[\cite{kun.obrien.ea:polynomial,czerwinski.nadara.ea:improved}]\label{kun_obrien_general}
  For any graph $G$, $\chicen(G)\le (\chilin(G))^{19}\cdot(\log(\chilin(G)))^{O(1)}$.
\end{thm}

\citet{kun.obrien.ea:polynomial} construct a family of graphs that contains, for every $\epsilon > 0$, a graph $G$ with $\chicen(G)\ge (2-\epsilon)\chilin(G)$. They conjecture that this bound is tight:

\begin{conj}[\cite{kun.obrien.ea:polynomial}]\label{crazy_conjecture}
  For every graph $G$, $\chicen(G)\le 2\chilin(G)$.
\end{conj}

This is a very bold conjecture since until now the only class of graphs for which a linear bound is known is the class of bounded degree trees. Specifically, for any tree $T$ of maximum-degree $\Delta\ge 3$, $\chicen(T) \le (\log_2(\Delta))\chilin(G)$ \cite[Theorem~4]{kun.obrien.ea:polynomial}.

To prove \cref{kun_obrien_general}, \citet{kun.obrien.ea:polynomial} establish the critical role that lower bounds on the linear chromatic number of pseudogrids and subcubic trees play in establishing an upper bound on $\chicen(G)$ as a function of $\chilin(G)$. In their work, they establish (asymptotically) tight lower bounds for the linear chromatic number of subcubic trees, but their lower bounds for pseudogrids are not tight.

With the goal of better understanding the difficulty of \cref{crazy_conjecture}, our objective in this paper is to establish a tight lower bound on the linear chromatic number of pseudogrids. To put our results into context and to be more precise, we first summarize the proof of \cref{kun_obrien_general}:

\begin{enumerate}
  \item A theorem of \citet{czerwinski.nadara.ea:improved} shows that, if $\chicen(G)\ge k^{19}\log^{q} k$ then $G$ contains a subcubic tree of treedepth $\Omega(k)$, or $\tw(G)\in\Omega(k^{18}\log^q k)$.\footnote{The result of \citet{czerwinski.nadara.ea:improved} is, of course, more general:  If $\chicen(G)\ge a\times b$, then $\tw(G)\in\Omega(a)$ or $G$ contains a subcubic tree of treedepth $\Omega(b)$. This is just an application of their result with $a=k^{18}\log^q k$ and $b=k$.}  In the former case, an asymptotically optimal result of \citet{kun.obrien.ea:polynomial} on subcubic trees completes the proof, so we are left with the case where $\tw(G)\in\Omega(k^{18}\log^q k)$.

  \item The current-best version of the Excluded Grid Theorem due to \citet{chuzhoy.tan:towards} shows that, if $\tw(G)\in\Omega(k^{18}\log^q k)$ (for a particular fixed positive $q$), then $G$ contains an $\Omega(k^2)\times \Omega(k^2)$ grid minor (equivalently, $G$ contains a $\Omega(k^2)\times\Omega(k^2)$ pseudogrid as a subgraph).

  \item Points~1 and 2 demonstrate that in order to establish \cref{kun_obrien_general}, a lower bound on the linear chromatic number of a $\Omega(k^2)\times\Omega(k^2)$ grid minor (i.e., pseudogrid) is needed.  \citet[Lemma~5]{kun.obrien.ea:polynomial} establish such a lower bound. Specifically, they show that, for any $\Omega(k^2)\times \Omega(k^2)$ pseudogrid $G$, $\chilin(G)\in\Omega(k)$.
  In the current work, we spend considerable effort to prove the following tight bound:
\end{enumerate}

\begin{lem}\label{pseudogrid_lower_bound}
  For any $k\times k$ pseudogrid $G$, $\chilin(G)\in\Omega(k)$.
\end{lem}

This improves the exponent in \cref{kun_obrien_general} from $19$ to $10$, yielding the following improvement to \cref{kun_obrien_general}:

\begin{thm}\label{kun_obrien_general2}
  For any graph $G$, $\chicen(G)\in (\chilin(G))^{10}\log(\chilin(G))^{O(1)}$.
\end{thm}

In addition to the improvement on the exponent in \cref{kun_obrien_general}, \cref{pseudogrid_lower_bound} adds further evidence in support of \cref{crazy_conjecture} by establishing that, when $G$ is a $k\times k$ pseudogrid, $\chicen(G)\in\Theta(\chilin(G))$.

Any further improvement to \cref{kun_obrien_general2} will either require an improved Excluded Grid Theorem or an entirely new approach. However, no improvement to the Excluded Grid Theorem will sufficient to establish a linear relationship between the centred and the linear chromatic number.  Indeed, the best possible Excluded Grid Theorem would state that any graph of treewidth $k^2\log k$ contains an $\Omega(k)\times\Omega(k)$ grid minor \cite{robertson.seymour.ea:quickly}, and the preceding argument would only show that $\chicen(G)\in O((\chilin(G))^3\log(\chilin(G)))$.  Even the Excluded Grid Theorem for Planar Graphs states that any planar graph of treewidth $k$ contains an $\Omega(k)\times\Omega(k)$ grid minor
\cite{robertson.seymour.ea:quickly}.
Combining this with the argument above and \cref{pseudogrid_lower_bound} shows only that, for any planar graph $G$, $\chicen(G)\in\Theta((\chilin(G))^2)$.

\section{Preliminaries}

In this paper, all graphs are simple and undirected. For a graph $G$, $V(G)$ denotes the vertex set of $G$, $E(G)$ denotes the edge set of $G$ and $\VE(G):=V(G)\cup E(G)$ denotes the set of vertices and edges of $G$. We will usually refer to an arbitrary element/edge/vertex in $\VE(G)$ as an \defin{object}. For a vertex $v\in V(G)$,  $N_G(v):=\{w\in V(G):vw\in E(G)\}$  denotes the open neighbourhood of $v$ in $G$ and for any set $S\subseteq V(G)$, $N_G(S):=\{w\in V(G):vw\in E(G),\, v\in S,\, w\not\in S\}$.  We use $\deg_G(v):=|N_G(v)|$ to denote the \defin{degree} of the vertex $v$ in the graph $G$.

For a  bipartite graph $H$, the two parts of $V(H)$ are denoted by $L(H)$ and $R(H)$ and we use the convention of writing an edge $xy$ so that its first endpoint $x$ is in $L(H)$ and its second endpoint $y$ is in $R(H)$.  A \defin{matching} $M$ in a bipartite graph $H$ is a subgraph of $H$ in which each vertex has degree at most $1$.  We say that $M$ \defin{saturates} a set $S\subseteq V(H)$ if $\deg_M(v)=1$ for each $v\in S$.  We make use of (the difficult half of) Hall's Marriage Theorem (see, for example \citet[Theorem~2.1.2]{diestel:graph}):

\begin{thm}[\citet{hall:on}]\label{hall}
  Let $H$ be a bipartite graph with the property that $|N_H(A)|\ge |A|$ for each $A\subseteq L(H)$.  Then $H$ contains a matching that saturates $L(H)$.
\end{thm}

We make use of the following (polygamous) consequence of Hall's Marriage Theorem:
\begin{cor}\label{d_hall}
  Let $d\ge 1$ be an integer and let $H$ be a bipartite graph with the property that $|N_H(A)|\ge d|A|$ for each $A\subseteq L(H)$.  Then $H$ contains a subgraph $M$ such that $\deg_M(v)=d$ for each $x\in L(H)$ and $\deg_M(y)\le 1$ for each $y\in R(H)$.
\end{cor}

\cref{d_hall} can be deduced from \cref{hall} by adding $d-1$ \defin{twins} $x_2,\ldots,x_d$ for each vertex $x\in L(H)$, i.e., by applying \cref{hall} on the graph $H':=H\cup\{x_iy:x\in L(H),\, i\in\{2,\ldots,d\},\, y\in N_H(x)\}$.

Late in the game, we will make use of the following asymmetric version of the Lovász Local Lemma (see, for example, \citet[Lemma~5.1.1]{alon.spencer:probabilistic}):

\begin{lem}\label{weighted_lovasz}
  Let $\mathcal{E}:=\{E_1,\ldots,E_n\}$ be a set of events in some probability space $(\Omega,\Pr)$.  For each $i\in\{1,\ldots,n\}$, let $\Gamma_i\subseteq \mathcal{E}$ be such that the event $E_i$ is mutually independent of $\mathcal{E}\setminus \Gamma_i$,\footnote{An event $A$ is mutually independent of a set $\{B_1,\ldots,B_r\}$ of events if, for any disjoint sets $I,J\subseteq\{1,\ldots,r\}$, $\Pr(A\cap\bigcap_{i\in I} B_i\cap\bigcap_{j\in J} \overline{B}_j=\Pr(A)\Pr(\bigcap_{i\in I} B_i\cap\bigcap_{j\in J} \overline{B}_j)$.} and let $w:\mathcal{E}\to[0,1)$ be such that
  \[
      \Pr(E_i) \le w(E_i)\cdot\prod_{E_j\in\Gamma_i}(1-w(E_j))  \enspace ,
  \]
  for each $i\in\{1,\ldots,n\}$.
  Then $\Pr(\overline{E}_1\cap\cdots\cap\overline{E}_n) > 0$.
\end{lem}

\section{The Linear Chromatic Number of Pseudogrids}

For positive integers $a$ and $b$, the \defin{$a\times b$ grid} $G_{a\times b}$ is the graph with vertex set $V(G_{a\times b}):=\{1,\ldots,a\}\times\{1,\ldots,b\}$ and that contains an edge with endpoints $(i_1,j_1)$ and $(i_2,j_2)$ if and only if $|i_1-i_2|+|j_1-j_2|=1$.  Such an edge is \defin{vertical} if $i_1=i_2$ and \defin{horizontal} if $j_1=j_2$.  For each $i\in\{1,\ldots,a\}$, the \defin{$i$th column} of $G_{a\times b}$ is the vertex set $\{(i,1),\ldots,(i,b)\}$ and, for each $j\in\{1,\ldots,b\}$, the \defin{$j$th row} is the vertex set $\{(1,j),\ldots,(a,j)\}$.  For any integer $0\le r<\min\{a,b\}/2$, the \defin{r-interior} of $G_{a\times b}$ defined as $\interior_r(G_{a\times b}):=G_{a\times b}[\{1+r,\ldots,a-r\}\times\{1+r,\ldots,b-r\}]$.  For $r\ge\min\{a,b\}/2$, $\interior_r(G_{a\times b})$ is the empty graph.

\begin{figure}
  \begin{center}
    \includegraphics{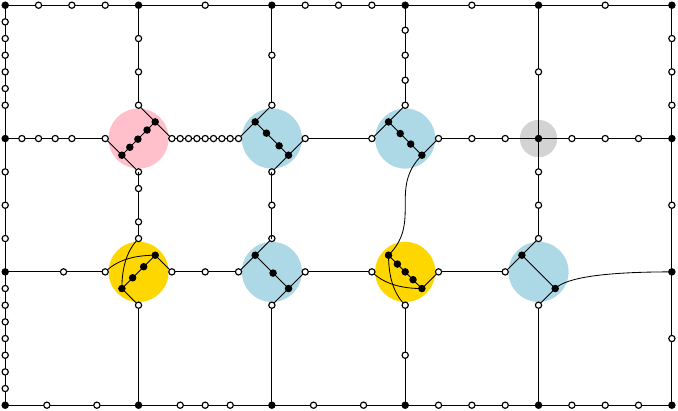}
  \end{center}
  \caption{A $6\times 4$ pseudogrid.  Black vertices are contained in $P_v$ for some vertex $v$ of $G_{6\times 4}$ and white vertices are contained in $P_{vw}$ for some edge of $G_{6\times 4}$. Colour are used to distinguish between cases (Q1) (blue), (Q2) (pink), and (Q3) (gold).}
  \label{pseudogrid_fig}
\end{figure}

Refer to \cref{pseudogrid_fig}.  An \defin{$a\times b$-pseudogrid} is any graph that can be obtained from $G_{a\times b}$ in the following way:
\begin{compactitem}
  \item Replace each degree-$4$ vertex $v$ whose neighbours (in counterclockwise order, starting with the neighbour $w$ above $v$) are $w,x,y,z$ with a non-empty path $P_v$. If $P_v$ has only one vertex then the unique vertex of $P_v$ is adjacent to each of $w,x,y,z$ and this has no effect on the underlying graph.  Otherwise, $P_v$ has two endpoints $p$ and $q$, each of which is adjacent to two vertices among $w,x,y,z$.  It is useful to consider three possible cases, each of which appears at least once in \cref{pseudogrid_fig}:
  \begin{compactenum}[(Q1)]
    \item \label{q_i} $p$ is adjacent to $\{w,x\}$ and $q$ is adjacent to $\{y,z\}$;
    \item \label{q_ii} $p$ is adjacent to $\{x,y\}$ and $q$ is adjacent to $\{w,z\}$; or
    \item \label{q_iii} $p$ is adjacent to $\{w,y\}$ and $q$ is adjacent to $\{x,z\}$.
  \end{compactenum}
  \item At this point, each edge $vw$ of $G_{a\times b}$ has a corresponding edge $v'w'$ in the modified graph, and we replace $v'w'$ with a path $\overline{P}_{vw}$ whose endpoints are $v'$ and $w'$.  (In other words, $\overline{P}_{vw}$ is a path obtained by subdividing the edge $v'w'$ zero or more times.)
\end{compactitem}

Let $G$ be an $a\times b$ pseudogrid.  For an edge $vw$ of $G_{a\times b}$, we let $P_{vw}:=\overline{P}_{vw}-\{v,w\}$ denote the (possibly empty) subpath containing the internal vertices of $\overline{P}_{vw}$.  For each vertex $v$ of $G_{a\times b}$ of degree less than $4$ we define $P_{v}$ to be the $1$-vertex path that contains only $v$.  In this way, $\mathcal{P}:=\{V(P_\mu):\mu\in \VE(G_{a\times b})\}$ is a partition of $V(G)$ into induced paths.  We call $\mathcal{P}$ a \defin{grid-partition} of $G$.
The \defin{$r$-interior} of $G$ is $\interior_r(G):=\bigcup_{\mu\in\VE(\interior_r(G_{a\times b}))} V(P_\mu)$.

Each row $R':=v_1,\ldots,v_a$ of $G_{a\times b}$ corresponds naturally to a path $R$ of $G$. The path $R$ contains $V(\overline{P}_{v_iv_{i+1}})$ for each $i\in\{1,\ldots,a-1\}$.  However, for $i\in\{2,\ldots,a-1\}$ $R$ may or may not contain $V(P_{v_i})$.  In particular, if $P_{v_i}$ was created using (Q\ref{q_iii}) then $R$ does not contain any internal vertices in $P_{v_i}$. Similarly, a column $C':=v_1,\ldots,v_b$ of $G_{a\times b}$ corresponds to a path $C$  in $G$ that contains $V(\overline{P}_{v_jv_{j+1}})$ for each $j\in\{1,\ldots,b-1\}$. This correspondence allows us to talk about the \defin{rows} and \defin{columns} of $G$, which we will do immediately.

As part of our proof, we use the operation of \defin{deleting} a row (or column) of $G$.  To delete a row $R$ of $G$ that corresponds to the row $R':=v_1,\ldots,v_a$ in $G_{a\times b}$, we remove the edges of $\overline{P}_{v_{i}v_{i+1}}$ for each $i\in\{2,\ldots,a\}$.  If this produces vertices of degree $1$ (which happens when $R$ is the first or last row of $G$ or when $R=v_1,\ldots,v_r$ does not contains $P_{v_i}$ for some $i\in\{1,\ldots,r\}$) then we repeatedly remove vertices of degree at most $1$ until none remain.  If $G$ is an $a\times b$ pseudogrid and we delete some row $R$, then the resulting graph is an $a\times (b-1)$ pseudogrid.  Similarly, if we delete column $C$ of $G$, then the resulting graph is a $(a-1)\times b$ pseudogrid.

\subsection{Proof Outline}

If some graph contains a $k\times k$ grid minor then it contains a $k\times k$ pseudogrid as a subgraph \cite{kun.obrien.ea:polynomial}.  Therefore to prove \cref{kun_obrien_general2}, it suffices to establish \cref{pseudogrid_lower_bound}.  We do this by showing that, for sufficiently small $\epsilon >0$, any $\epsilon k$-colouring of any $k\times k$ pseudogrid $G$ contains an uncentred path $P$. We prove the existence of $P$ in several steps; see \cref{outline}:

\begin{figure}
  \begin{center}
    \begin{tabular}{ccc}
      \includegraphics{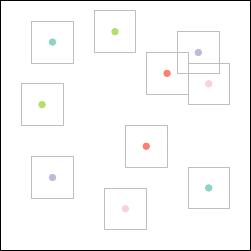} &
      \includegraphics{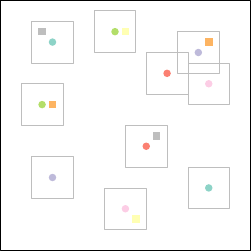} &
      \includegraphics{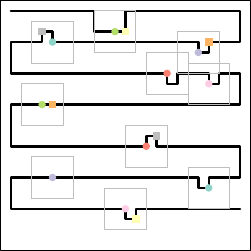} \\
      $S_1$ & $S_1\cup S_2$ & $P$
    \end{tabular}
  \end{center}
  \caption{Constructing the sets $S_1$ and $S_2$ and the path $P$.}
  \label{outline}
\end{figure}
\begin{compactenum}

  \item We first delete rows and columns from $G$ so that each colour that appears in $G$ appears at least $d$ times in the interior of $G$, for some large constant $d$.  From this point on the goal is to construct $P$ so that it contains each remaining colour at least twice.

  \item We greedily choose a set $S_1$ of vertices in $G$ that contains two vertices of each colour and that is `well-separated' in the sense that the corresponding set of vertices/edges in $G_{k\times k}$ have a minimum distance between them.  In \cref{outline} this minimum distance corresponds to the fact that the box drawn centred at each vertex in $S_1$ contains no other vertices of $S_1$.
  Unfortunately, this process can fail for some subset of the colours that appear in $G$.

  \item For these failed colours, we use \cref{d_hall} (the Polygamous Marriage Theorem) and \cref{weighted_lovasz} (the Lovász Local Lemma) to identify a set $S_2$ that contains two vertices of each of the missing colours and such that no vertex of $S_2$ is close to any other vertex of $S_2$ and each vertex of $S_1$ is close to at most one vertex of $S_2$.

  \item We construct a path $P$ that contains each vertex in $S_1\cup S_2$.  This is possible because each vertex in $S_1$ is `close to' at most one vertex of $S_2$ and vice-versa.  Aside from these pairs, no pair of vertices is close to each other.
\end{compactenum}

The most challenging aspect of this proof is the construction of $S_2$, which requires the use of the Local Lemma (\cref{weighted_lovasz}) to ensure that no vertex chosen to take part in $S_2$ is close to any other vertex in $S_2$.  The difficulty is illustrated in \cref{outline} by the cluster of three points of $S_1$ in the top right corner whose boxes overlap.  These vertices of $S_1$ are well-separated, but choosing one point from each of the three boxes to take part in $S_2$ could result in three vertices of $S_2$ being very close to each other.  In particular, these three points could be vertices of $P_{\mu_1}$, $P_{\mu_2}$ and $P_{\mu_2}$ where $\mu_1$, $\mu_2$, and $\mu_3$ are objects in $\VE(G_{k\times k})$ that all contain a common vertex, making it difficult or impossible to find a single path that contains all three (see \cref{bad_examples}).

\begin{figure}
  \begin{center}
    \begin{tabular}{ccc}
      \includegraphics{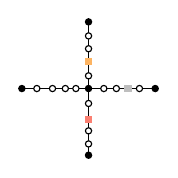} &
      \includegraphics{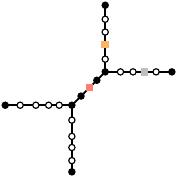} &
      \includegraphics{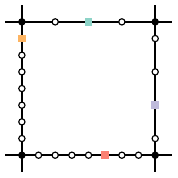}
    \end{tabular}
  \end{center}
  \caption{Some examples of $3$- and $4$-element vertex subsets that cannot all be contained in a single path (that must also contain other vertices not pictured).}
  \label{bad_examples}
\end{figure}

To avoid this, we take a random permutation $\pi$ of the vertices in $S_1$ and process them in order.  When a vertex is processed it \defin{claims} all the unmarked vertices in its box and then marks all the vertices in a larger box so that no subsequent vertex in $S_1$ is able to claim them.  The Local Lemma is then used to show that with positive probability, for each missing colour $\alpha$, there are many vertices of $S_1$ that claim a vertex of colour $\alpha$.  Finally, the Marriage Theorem is then used to show that we can find a matching between vertices in $S_1$ and claimed vertices so that each vertex of $S_1$ that takes part in the matching is matched with a vertex it claims and, for each missing colour $\alpha$, there are two vertices of colour $\alpha$ that take part in the matching.

The rest of this section expands the preceding sketch into a complete proof of \cref{pseudogrid_lower_bound}. In \cref{only_frequent_sec} we explain how to remove rows and columns of $G$ to obtain a sub-pseudogrid in which every colour appears frequently.  In \cref{path_finding} we explain how, given a well-separated set $S$ of vertices in $G$, to find a path that contains every vertex in $S$.  In \cref{packing_lemma_sec} we introduce a fairly standard-looking packing lemma that is needed in several places later.  In \cref{well_separated} we explain how to find a well-separated set $S$ that contains two vertices of each colour. Finally, \cref{wrapping_up} shows how to assemble these various pieces to prove \cref{pseudogrid_lower_bound}.

\subsection{Pseudogrids with Exclusively Frequent Colours}
\label{only_frequent_sec}

For the sake of compactness, let $G_k:=G_{k\times k}$.
Let $G$ be a $k\times k$ pseudogrid with grid partition $\mathcal{P}:=\{P_\mu:\mu\in\VE(G_{k})\}$ and let $\varphi:V(G)\to\{1,\ldots,c\}$ be a vertex colouring of $G$.  The partition $\mathcal{P}$ associates each vertex of $G$ with an edge or vertex $\mu$ of $G_k$, so $\varphi$ associates a colour set with each object in $\VE(G_k)$, as follows.  For each $\mu\in\VE(G_{k})$, we let $\varphi_{\mathcal{P}}(\mu):=\{\varphi(v):v\in V(P_\mu)\}$.  For each colour $\alpha\in\{1,\ldots,c\}$, define $\varphi^{-1}(\alpha):=\{v\in V(G):\varphi(v)=\alpha\}$ and define $\varphi_\mathcal{P}^{-1}(\alpha):=\{\mu\in \VE(G_{k\times k}):\alpha\in\varphi_\mathcal{P}(\mu)\}$.  For any colour set $A\subseteq\{1,\ldots,c\}$ define $\varphi^{-1}(A):=\bigcup_{\alpha\in A}\varphi^{-1}(\alpha)$ and $\varphi_\mathcal{P}^{-1}(A):=\bigcup_{\alpha\in A}\varphi_{\mathcal{P}}^{-1}(\alpha)$.  Throughout this section, we will use the idiom $\varphi(V(G))$ to denote the set of all colours used by $\varphi$ to colour the vertices of $G$.

The following lemma gives conditions that allow us to delete rows and columns from $G$ to obtain a sub-pseudogrid in which every colour occurs frequently in the interior.

\begin{lem}\label{only_frequent}
  Let $d,k,r\ge 1$ be integers, let $G$ be a $k\times k$ pseudogrid and let $\varphi$ be a vertex colouring of $G$ that uses $|\varphi(V(G))|\le k/(d+2r)$ colours.
  Then $G$ contains a $k'\times k'$ pseudogrid $G'$ with $k'\ge k - (d+2r)|\varphi(V(G))|$ that has a grid-partition $\mathcal{P}':=\{V(P'_\mu):\mu\in \VE(G_{k'})\}$ such that
  for any $A\subseteq \varphi(V(G'))$, $|\varphi_{\mathcal{P}'}^{-1}(A)\cap \interior_r(G_{k'})| \ge d|A|$.
\end{lem}

\begin{proof}
  The proof is by induction on $|\varphi(V(G))|$, the number of colours used by the colouring $\varphi$.  If there exists no $A\subseteq \varphi(V(G))$ with $|\varphi_{\mathcal{P}}^{-1}(A)\cap\interior_r(G_k)| < d|A|$ then taking $G':=G$ and $\mathcal{P}':=\mathcal{P}$ satisfies the requirements of the lemma.  Otherwise, fix such a set $A$. We will remove a set $R$ of rows and a set $C$ of columns from $G$ with $|R|=|C|\le d|A|+2r \le (d+2r)|A|$ to eliminate all vertices with colours in $A$, as follows:
  \begin{compactitem}
    \item For each vertex $v:=(i,j)\in V(\interior_r(G_{k}))$ with $A\cap\varphi_\mathcal{P}(v)\neq\emptyset$, we include row $j$ in $R$ and column $i$ in $C$.
    \item For each horizontal edge $vw\in E(\interior_r(G_{k}))$ with $A\cap\varphi_\mathcal{P}(vw)\neq\emptyset$, we include the row $R$ of $G$ that contains $P_{vw}$.
    \item For each vertical edge $vw\in E(\interior_r(G_{k}))$ with $A\cap\varphi_\mathcal{P}(vw)\neq\emptyset$, we include the column $C$ of $G$ that contains $P_{vw}$.
    \item We add the first and last $r$ rows to $R$ and the first and last $r$ columns to $C$.
    \item Finally, we add arbitrary rows to $R$ or columns to $C$ to ensure that $|R|=|C|$.
  \end{compactitem}
  At this point $|R|=|C|\le (d+2r)|A|$ and we remove all rows in $R$ and all columns in $C$ from $G$ to obtain a $k_0\times k_0$ pseudogrid $G_0$ with $k_0\ge k-(d+2r)|A|$ and such that $\varphi(V(G_0))\cap A=\emptyset$.  In particular, $|\varphi(V(G_0))|\le |\varphi(V(G))|-|A|$.

  Now apply induction on $G_0$ to get a $k'\times k'$ pseudogrid with
  \[
    k'\ge k_0-(d+2r)|\varphi(V(G_0))|
      \ge k-(d+2r)|A|-(d+2r)|\varphi(V(G_0))|
      \ge k - (d+2r)|\varphi(V(G))|
  \]
  that satisfies the conditions of the lemma.
\end{proof}

Having each object $\mu\in\VE(G_k)$ associated with a set $\varphi_\mathcal{P}(\mu)$ of colours rather than a single colour is problematic for what we want to to do next.  The following lemma allows us to choose one representative colour $\phi(\mu)$ from $\varphi_{\mathcal{P}}(\mu)$ for each $\mu\in \VE(G_k)$ while still ensuring that each colour appears frequently.

\begin{lem}\label{one_colour_per_object}
  Let $d,r>1$ be integers, let $G$ be a $k\times k$ pseudogrid with grid-partition $\mathcal{P}$, and let $\varphi$ be a vertex colouring of $G$ such that, for any $A\subseteq\varphi(V(G))$, $|\varphi_{\mathcal{P}}^{-1}(A)\cap\interior_r(G_k)| \ge d|A|$. Then there exists a colouring $\phi:\VE(G_{k})\to \varphi(V(G))\cup\{\perp\}$ with the following properties:
  \begin{compactenum}[(i)]
    \item $\phi(\mu)=\perp$ for each $\mu\not\in\VE(\interior_r(G_{k}))$;
    \item $\phi(\mu)=\perp$ or $\phi(\mu)\in\varphi_\mathcal{P}(\mu)$ for each $\mu\in\VE(G_{k})$; and
    \item $|\phi^{-1}(\alpha)|\ge d$ for each $\alpha\in\varphi(V(G))$.
  \end{compactenum}
\end{lem}

\begin{proof}
  Consider the bipartite graph $H$ with parts $X:=\varphi(V(G))$ and $Y:=\VE(\interior_r(G_k))$ and edge set
  \[
    E(H) := \{ (\alpha,\mu)\in X\times Y: \alpha\in\varphi_\mathcal{P}(\mu) \}
  \]
  By \cref{d_hall}, $H$ contains a subgraph $M$ with $\deg_M(\alpha)=d$ for each $\alpha\in X$ and $\deg_M(\mu)\le 1$ for each $\mu\in Y$.  For each edge $\alpha\mu$ in $M$, set $\phi(\mu):=\alpha$.  This defines $\phi(\mu)$ for any $\mu\in Y$ with $\deg_M(\mu)=1$.  For each $\mu\in Y$ with $\deg_M(\mu)=0$ set $\phi(\mu):=\perp$.
\end{proof}

\subsection{Finding Paths Through Well-Separated Pairs}
\label{path_finding}

Next we show that, given a sufficiently `well-separated' set $S$ of pairs of vertices in $G$, we can always find a path in $G$ that contains every vertex in $S$.
We base our definition of `well-separated' on the concept of boxes, which we now define.

The \defin{$r$-box} centred at a vertex $v:=(i,j)$ of $G_{k}$ is defined as
\[
  B_r(v) := \{i-r,\ldots,i+r\}\times\{j-r,\ldots,j+r\} \cap V(G_{k}) \enspace .
\]
The \defin{$r$-box} centred at an edge $vw$ of $G_{k}$ is $B_r(vw):=B_r(v)\cup B_r(w)$.\footnote{Technically the notations for $B_r(v)$ and $B_r(\mu)$ should include the value of $k$, but we omit this since there will never be any ambiguity.}
For any $\mu\in\VE(G_k)$, the $r$-box $B_r(\mu)$ defines an induced subgraph that we denote by $G_r(\mu):=G_{k}[B_r(\mu)]$.  Straightforward counting shows that, for any $\mu\in\VE(G_k)$,
\[
   |B_r(\mu)| \le |V(G_{2(r+1)\times (2r+1)})| = 2(r+1)(2r+1) = 4r^2+6r+2
\]
and
\begin{equation}
   |\VE(G_r(\mu))| \le |\VE(G_{2(r+1)\times (2r+1)})| = 12r^2+14r+3 \enspace . \label{rbox_size}
\end{equation}
For convenience, we define $\vol{r}:=12r^2+14r+3$ and the important thing to  keep in mind is that $\vol{r} \in\Theta(r^2)$.

We extend these definitions to vertices of a $k\times k$ pseudogrid $G$ with grid-partition $\mathcal{P}:=\{V(P_\mu):\mu\in\VE(G_k)\}$ as follows. For any $\mu\in\VE(G_k)$, define
\[
   \tilde{B}_r(\mu) := \bigcup_{\nu\in \VE(G_r(\mu))} V(P_\nu)
\]
and, for any $v\in V(G)$, let $\tilde{B}_r(v):=\tilde{B}_r(\mu_v)$ where $P_{\mu_v}$ is the unique part in $\mathcal{P}$ that contains $v$.\footnote{Technically, the notation for $\tilde{B}_r(v)$ should include the partition $\mathcal{P}$, but we omit this since there will never be any ambiguity as to which partition is being used.}
For any $v\in V(G)$, define $\tilde{G}_r(v):=G[\tilde{B}_r(v)]$.

The following lemma, whose proof is a case analysis that appears in \cref{pick_up_two_proof}, is the main tool we use to build a path that contains a set of vertices that can be paired off in such a way that each pair is far from all other vertices.

\begin{lem}[restate=pickuptwo,label=pick_up_two]
  Let $G$ be an $a\times a$ pseudogrid with $a\ge 5$,  let $s,v,w,t$ be vertices of $G$ with $s$ in column $1$ of $G$, $v,w\in \interior_1(G)$, and $t$ in column $a$ of $G$.  Then $G$ contains a path $P$ with endpoints $s$ and $t$ that contains $v$ and $w$.
\end{lem}

The next lemma shows how to take a well separated collection of pairs of vertices and cover them with disjoint boxes, each of which is compatible with \cref{pick_up_two}.

\begin{lem}\label{make_disjoint}
  Let $r,p$ be positive integers with $r-1\ge 4p+4$, let $G$ be a $k\times k$ pseudogrid and let $S\subseteq V(\interior_r(G))$ be such that $|\tilde{B}_{r}(v)\cap S|\le 2$ for each $v\in S$.
  Then there exists a set $X\subseteq V(G)$ such that
  \begin{compactenum}[(i)]
    \item \label{covers_s} $S\subseteq\bigcup_{x\in X}\tilde{B}_{p}(x)$;
    \item \label{two_per_box} $|\tilde{B}_{p}(x)\cap S|\le 2$ for each $x\in X$; and
    \item \label{disjoint_boxes} $\tilde{B}_{p+1}(x)\cap \tilde{B}_{p+1}(y)=\emptyset$ for each distinct $x,y\in X$.
  \end{compactenum}
\end{lem}

\begin{proof}
  For any integer $q$, say that two elements $v,w\in S$ are a \defin{$q$-pair} if $v\in \tilde{B}_{q}(w)$ or $w\in \tilde{B}_q(v)$.  First observe that any $v\in S$ takes part in at most one $(r-1)$-pair since, otherwise, $|\tilde{B}_{r}(v)\cap S|\ge 3$.

  Refer to \cref{make_disjoint_fig}.  We will define the set $X$ so that it satisfies (\ref{covers_s}) and for each $x\in X$ we will choose one or two elements of $\tilde{B}_{p}(x)\cap S$ and say that $x$ \defin{covers} those elements.  Let $S_1$ be the subset of $S$ containing only those elements that do not take part in any $(r-1)$-pair.  Let $X_1:= S_1$ and we say that each $x\in X_1$ covers itself.

  \begin{figure}
    \begin{center}
      \begin{tabular}{ccc}
        \includegraphics[width=.3\textwidth]{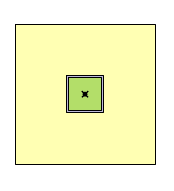} &
        \includegraphics[width=.3\textwidth]{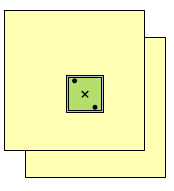} &
        \includegraphics[width=.3\textwidth]{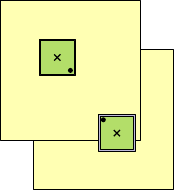} \\
        $X_1$ & $X_2$ & $X_3$
      \end{tabular}
    \end{center}
    \caption{The proof of \cref{make_disjoint}. Disks are vertices in $S$ and crosses are the resulting elements of $X$. The large yellow boxes are $B_{r-1}(v)$ for some $v\in S$ and the small mauve and green boxes are $B_{p+1}(x)$ and $B_{p}(x)$, respectively, for some $x\in X$.}
    \label{make_disjoint_fig}
  \end{figure}
  Let $S_2$ be the subset of $S\setminus S_1$ containing those vertices that take part in a $2p$-pair.  For each $2p$-pair $v,w\in S_2$, there exists  $x\in V(\interior_r(G))$ such that $v,w\in \tilde{B}_{p}(x)$. We include $x$ in $X_2$ and say that $x$ \defin{covers} $v$ and $w$.

  All that remains are the vertices in $S_3:=S\setminus (S_1\cup S_2)$ which take part in some $(r-1)$-pair but do not take part in any $2p$ pair.  For each $(r-1)$-pair $v,w\in S_3$, there exists two vertices $x,y\in V(G)$ such that $v\in \tilde{B}_{p}(x)$, $w\in\tilde{B}_{p}(y)$ and $\tilde{B}_{p+1}(x)\cap \tilde{B}_{p+1}(y)=\emptyset$.  We include $x$ and $y$ in $X_3$ and say that $x$ covers $v$ and $y$ covers $w$.

  Let $X:=X_1\cup X_2\cup X_3$.  By construction, $X$ satisfies (\ref{covers_s}), so it only remains to show that $X$ satisfies (\ref{two_per_box}) and (\ref{disjoint_boxes}).  That $X$ satisfies (\ref{two_per_box}) follows from (\ref{disjoint_boxes}), the fact that each element of $S$ is covered by some $x\in X$, and the fact that each $x\in X$ covers at most two elements of $S$.

  To see that $X$ satisifies (\ref{disjoint_boxes}), observe that if $x\in X$ covers $v\in S$ then $\tilde{B}_{p+1}(x)\subseteq \tilde{B}_{2p+2}(v)$. Therefore, if $x\in X$ covers $v\in S$,  $y\in X$ covers $w\in S$, and $\tilde{B}_{p+1}(x)\cap \tilde{B}_{p+1}(y)\neq\emptyset$ then $w\in \tilde{B}_{4p+4}(v)$ or $v\in \tilde{B}_{4p+4}(w)$.  Since $r-1 \ge 4p+4$, this implies that $v=w$ or $v$ and $w$ are an $(r-1)$-pair.  In the former case, $x=y$ and there is nothing to prove. In the latter case, $v,w\in S_2$ or $v,w\in S_3$.  If $v,w\in S_2$ then, again $x=y$ and there is nothing to prove.  If $v,w\in S_3$ then $x$ and $y$ are specifically chosen so that $\tilde{B}_{p+1}(x)$ and $\tilde{B}_{p+1}(y)$ are disjoint.
\end{proof}

\begin{lem}\label{pick_up_everything}
  Let $r\ge 9$, let $G$ be a $k\times k$ pseudogrid with grid partition $\mathcal{P}:=\{P_\mu:\mu\in\VE(G_k)\}$ and let $S\subseteq V(\interior_r(G))$ be such that $|\tilde{B}_{r}(v)\cap S|\le 2$ for each $v\in S$. Then $G$ contains a path that contains every vertex in $S$.
\end{lem}

\begin{proof}
  Apply \cref{make_disjoint} with $p=\floor{(r-5)/4}$ to obtain the set $X$.  Now, consider a `snake-like' path $P_0$ that contains every row $i$ with $i\equiv 1\pmod{2p+1}$ and, aside from those rows contains only vertices of columns $1$ and $k$.   For each $x\in X$, this path intersects exactly one row of $G[\tilde{B}_{p+1}(x)]$.  For each $x\in X$ we use \cref{pick_up_two} to replace $V(P_0)\cap \tilde{B}_{p+1}(x)$ with a path that is contained in $\tilde{B}_{p+1}(x)$ and contains the (at most two) vertices of $\tilde{B}_{p+1}(x)\cap S$.  After doing this for each $x\in X$ we obtain a path $P$ that contains every vertex in $S$.
\end{proof}

\subsection{A Packing Lemma}
\label{packing_lemma_sec}

We will make use of the following fairly standard looking packing lemma.

\begin{lem}\label{packing_lemma}
  Let $r\ge 10$ be an even integer and let $Q\subseteq \VE(G_k)$ be such that $\mu_1\not\in \VE(G_{2r+1}(\mu_2))$ for each distinct $\mu_1,\mu_2\in Q$.  Then, for any $\mu\in \VE(G_k)$,
  $|\{\mu_1\in Q: \mu\in \VE(G_{3r+1}(\mu_1))\}| \le 16$.
\end{lem}

\begin{proof}
  To avoid the need to discuss boundary conditions that have no effect on the result, it is simpler to consider points that are sufficiently far from the boundary of $G_k$. More precisely, we may assume that $Q\cup\{\mu\}\subseteq\interior_{4r+2}(G_{k+8r+4})$.

  The packing condition that $\mu_1\not\in \VE(G_{2r+1}(\mu_2))$ for each distinct $\mu_1,\mu_2\in Q$ implies that $\VE(G_{r}(\mu_1))\cap \VE(G_{r}(\mu_2))=\emptyset$.  Let $Q':=\{\mu_1\in Q: \mu\in \VE(G_{3r+1}(\mu_1))\}$ so that our task is to show that $|Q'|\le 16$.  For any $\mu_1\in Q'$,
  $\mu_1\in\VE(G_{3r+2}(\mu))$ and consequently, $\VE(G_{r}(\mu_1))\subseteq \VE(G_{4r+2}(\mu))$.  For each $\mu_1\in Q'$, $|\VE(G_r(\mu_1))|= 12r^2+8r+1$.  Therefore,
  \[  |Q'|\le \frac{\vol{4r+2}}{12r^2+8r+1} < 17 \enspace ,
  \]
  for all $r\ge 10$.
\end{proof}

\subsection{Finding a Well-Separated Set}
\label{well_separated}

Next we show how, given a colouring $\varphi$ like that guaranteed by \cref{only_frequent}, we can find a set of vertices in $G$ that contain two vertices of each colour and that is compatible with \cref{pick_up_everything}.

\begin{lem}\label{doubled_colour_set}
  For every integer $r\ge 10$ there exists an integer $d\in O(r^4\log r)$ such that the following is true, for every integer $k\ge 1$.\newline
  Let $G$ be a $k\times k$ pseudogrid with grid-partition $\mathcal{P}:=\{V(P_\mu):\mu\in\VE(G_k)\}$ and let $\varphi$ be a vertex colouring of $G$ with the property that, for each $A\subseteq\varphi(V(G))$, $|\varphi_{\mathcal{P}}^{-1}(A)\cap\interior_r(G_k)|\ge d|A|$.
  Then there exists $S\subseteq V(\interior_r(G))$ such that
  \begin{compactenum}[(i)]
    \item \label{hits_both} $|\varphi^{-1}(\alpha)\cap S|= 2$ for each $\alpha\in\varphi(V(G))$; and
    \item \label{spread_out} $|\tilde{B}_r(v)\cap S|\le 2$ for each $v\in S$.
  \end{compactenum}
\end{lem}

\begin{proof}
  We begin by applying \cref{one_colour_per_object} to obtain the colouring $\phi:\VE(G_k)\to\varphi(V(G))\cup\{\perp\}$ in which $\phi^{-1}(\alpha)\ge d$ for each $\alpha\in\varphi(V(G))$.  Observe that it is now sufficient to find a $2|\varphi(V(G))|$-element subset $Q\subseteq\VE(\interior_r(G_k))$ such that
  \begin{compactenum}[(i)]
    \item $|\phi^{-1}(\alpha)\cap Q|= 2$ for each $\alpha\in\varphi(V(G))$ and
    \item $|\VE(G_r(\mu))\cap Q|\le 2$ for each $\mu\in Q$.
  \end{compactenum}
  Indeed, with such a $Q$ we can obtain $S$ by taking one vertex of colour $\phi(\mu)$ from $P_\mu$ for each $\mu\in Q$.

  We construct $Q$ in two rounds.  In the first round, we start with an initially empty set $Q_1$ and repeat the following for each $\alpha\in\varphi(V(G))$:
  If there exists distinct $\mu_1,\mu_2\in\VE(G_k)$ such that
  \begin{compactenum}[(a)]
    \item $\phi(\mu_1)=\phi(\mu_2)=\alpha$; and\label{hits_both_q}
    \item for each $b\in\{1,2\}$, $\mu_b\not\in \VE(G_{2r+1}(\mu_{3-b})) \cup \bigcup_{\mu\in Q_1} \VE(G_{2r+1}(\mu))$,\label{spread_out_q}
  \end{compactenum}
  then we include $\mu_1$ and $\mu_2$ in $Q_1$ and declare the first round an \defin{$\alpha$-success}.  Otherwise, we declare the first round an \defin{$\alpha$-failure}.

  Let $\overline{X}\subseteq\varphi(V(G))$ be the set of colours $\alpha$ for which the first round was an $\alpha$-success. At the end of this process, the set $Q_1$ certainly satisfies (\ref{spread_out}).  In fact, it satisfies an even stronger property: $B_{2r+1}(\mu)\cap Q_1=\{\mu\}$ for each $\mu\in Q_1$, which implies that $|B_r(\nu)\cap Q_1|\le 1$ for each $\nu\in\VE(G_k)$. However, $Q_1$ only satisfies (\ref{hits_both}) for the colours in $\overline{X}$.  We now use the second round to create a set $Q_2$ to handle the colours in $X:=\varphi(V(G))\setminus \overline{X}$.

  Define a bipartite graph $H$ with parts $L(H):=X$ and $R(H):=Y:=Q_1$.  Let $\pi:=\mu_1,\ldots,\mu_{|Y|}$ be a random permutation of $Y$.\footnote{A \defin{random permutation} of $Y$ is a permutation chosen uniformly from the set of all $|Y|!$ permutations of $Y$.}  We include an edge $\alpha\mu_i$ in $H$ if and only if
  \[
    \VE(G_{2r+1}(\mu_i))\setminus\left(\bigcup_{j=1}^{i-1}\VE(G_{3r+1}(\mu_j))\right)
  \]
  contains some object $\nu$ of colour $\phi(\nu)=\alpha$.  A helpful way to view this process is as follows: Initially, every vertex of $G_k$ is \defin{unmarked}.  For $i=1,\ldots,|Y|$, $\mu_i$ claims every unmarked object in $\VE(G_{2r+1}(\mu_i))$ and then marks every object in $\VE(G_{3r+1}(\mu_i))$. Then $H$ contains the edge $\alpha\mu_i$ if and only if $\mu_i$ claims at least one object of colour $\alpha$.

  We now want to use the Local Lemma (\cref{weighted_lovasz}) to show that, with probability greater than zero, $\deg_H(\alpha)\ge 2\vol{2r+1}$ for each $\alpha\in X$.  For each $\alpha\in X$, consider the set $\phi^{-1}(\alpha)\subseteq \VE(G_k)$ and recall that, since $\phi$ comes from \cref{one_colour_per_object}, $|\phi^{-1}(\alpha)| \ge d$.

  For each $\alpha\in X$, let $\Phi_\alpha:=\phi^{-1}(\alpha)\cap \left(\bigcup_{\mu\in Y}\VE(G_{2r+1}(\mu))\right)$.  In words, $\Phi_\alpha$ consists of all objects of colour $\alpha$ that were eliminated from consideration because they are too close to an element of $Y$. We claim that, for each $\alpha\in X$, $|\Phi_\alpha|\ge d-\vol{r}$.  Otherwise $|\phi^{-1}(\alpha)\setminus\Phi_\alpha| > \vol{r}$. Then, when the colour $\alpha$ was considered during the first round, we could have taken $\mu_1$ to be any object in $\phi^{-1}(\alpha)\setminus\Phi_\alpha$ and taken $\mu_2$ to be any object in $\phi^{-1}(\alpha)\setminus\Phi_\alpha\setminus B_{2r+1}(\mu_1)$. (Recall that $\vol{r}\ge |B_{2r+1}(\mu_1)|$, by definition.)  Thus, the first round would have been an $\alpha$-success, contradicting the fact that $\alpha\in X$.

  Our next step is to show that, for each $\alpha\in X$, the random variable $\deg_H(\alpha)$ dominates\footnote{A random variable $X$ dominates a random variable $Y$ if $\Pr(X\ge x)\ge \Pr(Y\ge x)$ for all $x\in\R$.} a $\binomial(\ceil{|\Phi_\alpha|/\vol{7r}},1/16)$ random variable, which allows us to establish an exponential inequality for $\Pr(\deg_H(\alpha) < 2\vol{2r+1})$. Suppose that, for some $\alpha\in X$ and $\mu\in Y$, $B_{2r+1}(\mu)\cap \Phi_\alpha$ contains at least one element $\nu$ and consider the set
  \[
     Y_\nu := \{\xi\in Y:\nu\in\VE(G_{3r+1}(\xi))\} \enspace .
  \]
  The edge $\alpha\mu$ will appear in $H$ if, in our random permutation $\pi$ of $Y$, $\nu$ appears before any other element of $Y_\nu$.  By \cref{packing_lemma}, $|Y_\nu|\le 16$, so the probability that the edge $\alpha\mu$ appears in $H$ is at least $1/|Y_\nu|\ge 1/16$. This is already enough to establish that $\E(\deg_H(\alpha)) \in \Omega(|\Phi_\alpha|/r^2)$, but in order to obtain a sufficiently strong concentration result for $\deg_H(\alpha)$ we need to find some independence.\footnote{There are many ways to obtain a concentration result here and we use a simple method with the goal of being self-contained.  For a tighter result, use the method of bounded differences \cite{mcdiarmid:on}.}  To do this, set $J:=\Phi_\alpha$ and consider the greedy process of repeatedly choosing some $\mu\in Y$ such that $\VE(G_{2r+1}(\mu))\cap J$ contains at least one element $\nu$ and then setting $J:=J\setminus \VE(G_{7r}(\nu))$.\footnote{The constant $7$ is overkill here and is only used for simplicity; $6r+O(1)$ would be sufficient.}  This process continues until $J$ is empty. Since $|\VE(G_{7r}(\mu))|\le \vol{7r}$, each iteration in this process eliminates at most $\vol{7r}$ elements from $J$ so the number, $t$, of iterations is at least $\ceil{|\Phi_\alpha|/\vol{7r}}$.

  Let $\{\mu_1',\ldots,\mu_t'\}$ be the subset of $Y$ chosen by this process and let $\{\nu_1,\ldots,\nu_t\}\subseteq\Phi_\alpha$ be the corresponding elements of $\Phi_\alpha$.  The important observation now is that the sets $Y_{\nu_1},\ldots,Y_{\nu_t}$ are pairwise disjoint.  Indeed, $Y_{\nu_i}\subseteq B_{3r+2}(\nu_i)$, $Y_{\nu_j}\subseteq B_{3r+2}(\nu_j)$ and the fact that $\nu_i\not\in B_{7r}(\nu_j)$ implies that $\VE(G_{3r+2}(\nu_i))\cap \VE(G_{3r+2}(\nu_j))=\emptyset$.  Therefore, if we let $U_i$ denote the event ``$\mu_i'$ apppears in $\pi$ before any other element of $Y_{\nu_i}$'' then the events $U_1,\ldots,U_t$ are mutually independent.  Indeed, each $U_i$ depends only on the relative order of the subset $Y_{\nu_i}$ in the permutation $\pi$.  For each $i\in\{1,\ldots,k\}$, $\Pr(U_i)\ge 1/16$.  Therefore, $\deg_H(\alpha)\ge \sum_{i=1}^t \mathbbm{1}_{U_i}$ dominates a $\binomial(t,1/16)$ random variable.  Therefore,
  \begin{align}
    \Pr\left(\deg_H(\alpha)< 2\vol{2r+1}\right)
    & \le \Pr\left(\sum_{i=1}^t \mathbbm{1}_{U_i} < 2\vol{2r+1}\right) \notag \\
    & \le \Pr(\binomial(t,1/16)\le 2\vol{2r+1}) \notag \\
    & = \sum_{x=0}^{2\vol{2r+1}} \binom{t}{x}\left(\frac{1}{16}\right)^x\left(\frac{15}{16}\right)^{t-x} \notag \\
    & < \sum_{x=0}^{2\vol{2r+1}} \binom{t}{x}\left(\frac{15}{16}\right)^{t} \notag \\
    & < (2t)^{2\vol{2r+1}}\cdot\left(\frac{15}{16}\right)^{t} \notag \\
    & = \exp\big({2\vol{2r+1}}\ln(2t)-t\ln(16/15)\big) \notag \\
    & \le \exp\left({2\vol{2r+1}}\ln(2|\Phi_\alpha|)-\frac{|\Phi_\alpha|\ln(16/15)}{\vol{7r}}\right)
      & \text{(since $|\Phi_\alpha|/\vol{7r} \le t \le |\Phi_\alpha|$)} \notag \\
    & \le \exp\left({2\vol{2r+1}}\ln(2|\Phi_\alpha|)-\frac{|\Phi_\alpha|}{16\vol{7r}}\right)
      & \text{(since $\ln(16/15) > 1/16$).}
       \label{probability}
  \end{align}

  We are now ready for an application of \cref{weighted_lovasz}.  For each $\alpha\in X$, let $E_\alpha$ denote the event ``$\deg_H(\alpha)< 2\vol{2r+1}$'' and let $\mathcal{E}:=\{E_\alpha:\alpha\in X\}$. For each $\alpha\in X$, define
  \[
     \Gamma_\alpha := \left\{E_\beta \in \mathcal{E}: \Phi_\beta\cap \left(\bigcup_{\nu\in \Phi_\alpha} \VE(G_{7r}(\nu))\right)\neq\emptyset\right\}
  \]
  and observe that $|\Gamma_\alpha|\le \vol{7r}\cdot|\Phi_\alpha|$.
  The event $E_\alpha$ is mutually independent of $\mathcal{E}\setminus\Gamma_\alpha$ since, for any $E_\beta\in\mathcal{E}\setminus\Gamma_\alpha$ the sets $Y_\alpha:=\bigcup_{\nu\in\Phi_\alpha} Y_\nu$ and $Y_\beta:=\bigcup_{\nu\in \Phi_\beta} Y_\nu$ are disjoint and $\deg_H(\alpha)$ and $\deg_H(\beta)$ are each entirely determined by the relative orders of $Y_\alpha$ and $Y_\beta$ in $\pi$, respectively.

  For each $\alpha\in X$, let $w(E_\alpha):=1/\tau$, for some $\tau>1$ to be defined shortly.  Then
  \begin{align}
    w(E_\alpha)\prod_{E_\beta\in\Gamma_\alpha} (1-w(E_\beta))
    & = \tfrac{1}{\tau}\left(1-\tfrac{1}{\tau}\right)^{|\Gamma_\alpha|-1} \notag \\
    & > \tfrac{1}{\tau}\left(1-\tfrac{1}{\tau}\right)^{|\Gamma_\alpha|}
      & \text{(since $1-1/\tau<1$)}\notag \\
    & > \exp(-\ln\tau-|\Gamma_\alpha|/\tau)
      & \text{(since $e^{-1/\tau}> 1-1/\tau$)}\notag \\
    & \ge \exp\left(-\ln\tau - \vol{7r}\cdot|\Phi_\alpha|/\tau\right)
      & \text{(since $|\Gamma_\alpha|\le \vol{7r}\cdot\Phi_\alpha$).}  \label{weight_product}
  \end{align}

  Comparing \cref{weight_product,probability} we see that both quantities decrease exponentially in $|\Phi_\alpha|$ but \cref{weight_product} does so at a rate that can be controlled with $\tau$.
  Taking $\tau = {32\vol{7r}^2}$, and using \cref{weight_product,probability}, we find that
  \[
     \Pr\left(E_\alpha\right) \le w(E_\alpha)\prod_{\beta\in\Gamma_\alpha}(1-w(E_\beta))
  \]
  provided that
  \[   |\Phi_\alpha| \ge {32\vol{7r}}\big(\ln({32\vol{7r}}+1) + {2\vol{2r+1}}\ln(2|\Phi_\alpha|)\big) \enspace .
  \]
  The right hand side of this last equation is of the form $O(r^2\log r+r^4\log|\Phi_\alpha|)$.  Since $|\Phi_\alpha|\ge d-\vol{r}\in d- O(r^2)$, this can therefore be satisfied for some $d\in O(r^4\log r)$.

  Therefore, by \cref{weighted_lovasz}, there exists a permutation $\pi$ of $Y$ that produces a bipartite graph $H$ with $\deg_H(\alpha)\ge {2\vol{2r+1}}$ for each $\alpha\in X$.  On the other hand, $\deg_{H}(y)\le |\VE(G_r(y))|-1 < {\vol{2r+1}}$ for each $y\in Y$.  Therefore, for any $A\subseteq X$, $|N_H(A)|> {2\vol{2r+1}}|A|/{\vol{2r+1}} = 2|A|$.  Therefore, by \cref{d_hall}, there is a subgraph $M$ of $H$ in which $\deg_M(\alpha)=2$ for each $\alpha\in X$ and $\deg_M(\mu)\le 1$ for each $\mu\in Y$. Each edge $\alpha\mu$ of $M$ corresponds to some $\nu\in\Phi_\alpha\cap\VE(G_r(\mu))$, which we place in $Q_2$.

  Now it is straightforward to verify that $Q:=Q_1\cup Q_2$ satisifies (\ref{hits_both}).  Furthermore, for each $\mu\in Q_1$, $B_r(\mu)$ contains no object in $Q_1\setminus\{\mu\}$ and contains at most one element of $Q_2$.  Similarly, for each $\nu\in Q_2$, $B_r(\nu)$ contains no element of $Q_2\setminus\{\nu\}$ and contains at most one element of $Q_1$.  Therefore $Q$ satisfies (\ref{spread_out}).
\end{proof}

\subsection{Wrapping Up}
\label{wrapping_up}

We now have all the pieces needed to complete our lower bound on the linear chromatic number of pseudogrids:

\begin{proof}[Proof of \cref{pseudogrid_lower_bound}]
  Let $G$ be a $k\times k$ pseudogrid and suppose, for the sake of contradiction, that $G$ has a linear colouring $\varphi$ using fewer than $\epsilon k$ colours where $\epsilon:=1/(d+2r)$, $r:= 13$ and $d\in O(r^4\log r)=O(1)$ is some large integer constant.  Then, by \cref{only_frequent}, $G$ contains a $k'\times k'$ pseudogrid $G'$ with $k'\ge k/2$ having a grid-decomposition $\mathcal{P}':=\{P_\mu':\mu\in \VE(G_{k'})\}$ with the property that, for any $A\subseteq\varphi(V(G'))$, $|\varphi^{-1}_{\mathcal{P'}}(A)\cap\interior_r(G_{k'})|\ge d|A|$.  For a sufficiently large constant $d$, \cref{doubled_colour_set} implies that $G'$ contains a set $S$ of $2|\varphi(V(G')|$ vertices containing two vertices of each colour in $\varphi(V(G'))$ such that $|B_r(\mu)\cap S|\le 2$ for each $\mu\in S$.  By \cref{pick_up_everything}, $G'$ contains a path $P$ that contains every vertex in $S$.  Since $P$ is contained in $G'$, $\varphi(V(P))=\varphi(V(G'))$, so $P$ is has no center under $\varphi$, contradicting the assumption that $\varphi$ is a linear colouring of $G$.
\end{proof}

\section*{Acknowledgement}

The main subject of this paper was posed as an open problem at the 9th Annual Bellairs Workshop on Geometry and Graphs (v2), which was held January 21--28, 2022 and again at the 16th Annual Workshop on Probability and Combinatorics, which was held March 25--April 1, 2022. We are grateful to the organizers and participants of both workshops for providing a stimulating research environment.
Pat~Morin especially enjoyed early-morning discussions with Stefan~Langerman on the Seabourne Terrace.

We are also grateful to Omer Angel for pointing out that the problem on a $k\times k$ grid (as opposed to pseudogrid) has a simple solution, illustrated in \cref{omer}: Partition the grid into $2\times 2$ grids and repeatedly delete a subgrid if it contains a vertex whose colour is unique.  If this causes the grid to become disconnected, keep only the largest remaining component and continue.  This process deletes at most $4\epsilon k$ vertices.  For $\epsilon < 1/4$, this implies that the graph that remains at the end of thie process has at least $k^2/2$ vertices.  In particular, the graph that remains is non-empty and every colour that appears in this graph is used on at least two vertices.  Finally, this graph has a Hamiltonian path (which can be constructed by traversing the boundary of a certain tree contained in the dual graph).  Unfortunately, this solution breaks down immediately for pseudogrids since a pseudogrid is not necessarily Hamiltonian, even before removing some of its vertices and edges.

\begin{figure}
  \centering
  \begin{tabular}{cccc}
    \includegraphics[page=1]{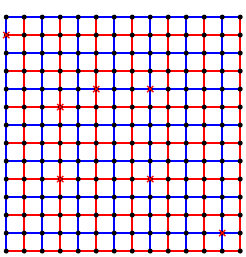} &
    \includegraphics[page=2]{figs/omer} &
    \includegraphics[page=4]{figs/omer}
  \end{tabular}
  \caption{A simpler solution for colourings of the grid. The (orange) tree used to create the (purple) Hamiltonian path is created so that its edges and leaves use only odd-numbered rows and columns of the dual graph.}
  \label{omer}
\end{figure}

\bibliographystyle{plainurlnat}
\bibliography{lin-vs-cen}

\appendix
\section{Proof of \cref{pick_up_two}}
\label{pick_up_two_proof}

\pickuptwo*

\begin{proof}[Proof of \cref{pick_up_two}]
  Let $\mathcal{P}:=\{P_\mu:\mu\in \VE(G_a)\}$ be a grid partition of $G$ and let $\mu_v,\mu_w\in \VE(G_a)$ be such that $v\in V(P_{\mu_v})$ and $w\in V(P_{\mu_w})$.  Define $\mu_s$ and $\mu_t$ similarly, with respect to $s$ and $t$.  Very roughly, this lemma says that the grid $G_a$ contains a path $P_a$ with $\mu_v,\mu_w\in\VE(P_a)$ whose first edge/vertex is $\mu_s$, and whose last edge/vertex is $\mu_t$.

  Although this is a simplification of the problem, we first describe how to solve it.  For this simpler problem, we may assume that each of $\mu_s$, $\mu_v$, $\mu_w$, and $\mu_t$ is an edge of $G_a$ since we can replace any vertex of $G_a$ with one of its incident edges. Refer to \cref{generic_path}. Beginning at $\mu_s$, $P_a$ can traverse the boundary of $G_a$ until reaching the first column $i$ that contains an endpoint of $\mu_v$ or $\mu_w$, then vertically in this column to collect $\mu_v$ (say). What happens next depends on whether or not $\mu_w$ also intersects column $i$.
  \begin{figure}[htbp]
    \begin{center}
      \begin{tabular}{cc}
        \includegraphics{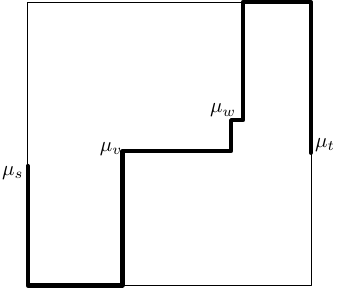} &
        \includegraphics{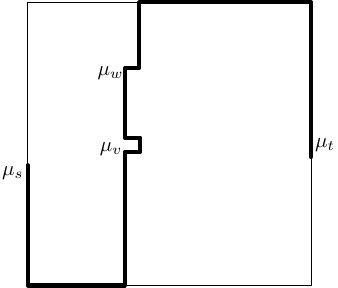}
      \end{tabular}
    \end{center}
    \caption{A simplification of \cref{pick_up_two}}
    \label{generic_path}
  \end{figure}
  \begin{compactitem}
    \item If $\mu_w$ also intersects column $i$, then $P_a$ can immediately return to column $i$ if necessary (if $\mu_v$ was a horizontal edge) and collect $\mu_w$. The only care that needs to be taken in this case occurs when $\mu_v$ and $\mu_w$ are a pair of horizontal and vertical edges that share an endpoint.  When this happens, the initial traversal on the boundary should be done (clockwise or counterclockwise) so that the vertical edge appears first in $P_a$.

    \item If $\mu_w$ does not intersect column $i$, then after collecting $\mu_v$, $P_a$ can proceed horizontally to the first column $j$ that intersects $\mu_w$ and then vertically in column $j$ to collect $\mu_w$.
  \end{compactitem}
  In either case, after collecting $\mu_w$, $P_a$ can then proceed vertically to return to the boundary and then traverse the boundary to finish at $\mu_t$.

  The rough description given above works perfectly when $\mu_v$ and $\mu_w$ are each edges of $G_a$.  However, some mild complications arise when one or both of $\mu_v$ or $\mu_w$ are vertices of $G_a$.  This is due to the fact that a path $P_a$ in $G_a$ that contains a vertex $\nu$ may not correspond to a path $P$ in $G$ that contains $P_\nu$.  Indeed, this depends on whether $P_a$ `turns' at $\nu$ and whether $P_\nu$ was created using (Q\ref{q_i}), (Q\ref{q_ii}), or (Q\ref{q_iii}).  We say that a $\mu_v$ or $\mu_w$ is \defin{straight} if it was created using (Q\ref{q_i}) or (Q\ref{q_ii}) and $\mu_v$ or $\mu_w$ is \defin{bent} if it was created using (Q\ref{q_iii}).  Without loss of generality, assume that $\mu_v$ is a vertex in column $i$ of $G_a$ and that $\mu_w$ does not intersect columns $2,\ldots,i-1$.  Our strategy is to replace $\mu_v$ (and possibly also $\mu_w$) with a pair of edges incident on $\mu_v$ in such a way that the algorithm described above is able to construct a path $P_a$ that contains the resulting collection of $3$ or $4$ edges.  What follows is a (boring) case analysis (see \cref{cases}):

  \begin{figure}
    \begin{center}
      \includegraphics{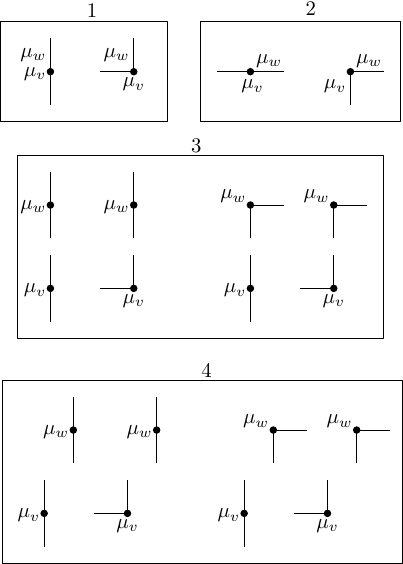}
    \end{center}
    \caption{Cases in the proof of \cref{pick_up_two}}
    \label{cases}
  \end{figure}

  \begin{compactenum}
    \item If $\mu_w$ is a vertical edge incident on $\mu_v$ then there are two possibilities:
    \begin{compactitem}
      \item If $\mu_v$ is straight then we replace $\mu_v$ with the other vertical edge incident on $\mu_v$.
      \item If $\mu_v$ is bent then we replace $\mu_v$ with the horizontal edge joining $\mu_v$ to a vertex in column $i-1$.
    \end{compactitem}
    \item If  $\mu_w$ is a horizontal edge incident to $\mu_v$ then there are two possibilities:
    \begin{compactitem}
      \item If $\mu_v$ is straight then we replace it with the other horizontal edge incident to $\mu_v$.
      \item If $\mu_v$ is bent then we replace it with one of the vertical edges incident to it.
    \end{compactitem}
    \item If $\mu_w$ intersects column $i$ but is not an edge incident to $\mu_v$ then we may assume, without loss of generality, that $\mu_w$ is above $\mu_v$.  There are two cases to consider:
    \begin{compactitem}
      \item If $\mu_v$ is straight, then we replace $\mu_v$ with the two vertical edges incident on $\mu_v$.
      \item If $\mu_v$ is bent, then we replace $\mu_v$ with the horizontal edge joining $\mu_v$ to a vertex in column $i-1$ and the vertical edge incident to $\mu_v$ and above $\mu_v$.
    \end{compactitem}
    If $\mu_w$ is an edge of $G_a$ then there is nothing else to do. If $\mu_w$ is a vertex of $G_a$ then there are two cases to consider:
    \begin{compactitem}
      \item If $\mu_w$ is straight, then we replace $\mu_w$ with the two vertical edges incident on $\mu_w$.
      \item If $\mu_w$ is bent, then we replace $\mu_w$ with the horizontal edge joining $\mu_v$ to a vertex in column $i+1$ and the vertical edge incident to $\mu_w$ and below $\mu_w$.
    \end{compactitem}
    \item If $\mu_w$ does not intersect column $i$ there are two possibilities:
    \begin{compactitem}
      \item If $\mu_v$ is straight, then we replace $\mu_v$ with the two vertical edges incident on $\mu_v$.
      \item If $\mu_v$ is bent, then we replace $\mu_v$ with the horizontal edge joining $\mu_v$ to a vertex in column $i-1$ and a vertical edge in column $i$.
    \end{compactitem}
    If $\mu_w$ is an edge of $G_a$ then there is nothing further to do. If $\mu_w$ is a vertex of $G_a$ in column $j>i$ then there are again two possibilities:
    \begin{compactitem}
      \item If $\mu_w$ is straight, then we replace $\mu_w$ with the two vertical edges incident on $\mu_w$.
      \item If $\mu_w$ is bent, then we replace $\mu_w$ with the horizontal edge joining $\mu_v$ to a vertex in column $j+1$ and a vertical edge in column $j$.
    \end{compactitem}
  \end{compactenum}
  Now, exactly the same strategy used above can be used to construct a path $P_a$ that contains the (up to four) required edges of $G_a$ and the corresponding path $P$ in $G$ satisfies the requirements of the lemma.
\end{proof}

\end{document}